\documentclass[letterpaper, 10 pt, conference]{ieeeconf}  

\IEEEoverridecommandlockouts                              
\usepackage[T1]{fontenc}
\usepackage{amssymb}
\usepackage{amsmath}
\usepackage{amsfonts}
\usepackage{times}
\usepackage{epsfig}
\usepackage{graphics}

\usepackage{cite}

\let\labelindent\relax
\usepackage{enumitem}
\usepackage{xcolor}
\usepackage[export]{adjustbox}
\usepackage[colorlinks, citecolor=blue, linkcolor=blue, urlcolor=blue]{hyperref}
\newtheorem{thm}{\bf Theorem} 

\newtheorem{rem}[thm]{\bf Remark}
\newtheorem{ass}[thm]{\bf Assumption}
\newtheorem{prop}[thm]{\bf Proposition}

\newtheorem{dfn}[thm]{\bf Definition}
\newtheorem{eg}[thm]{\bf Example}

\title{\LARGE \bf
Verifying Probabilistic Regions of Attraction with Neural Lyapunov Functions for Stochastic Systems
}

\author{Yun Su, Hans De Sterck, and Jun Liu  
\thanks{This work was supported in part by the NSERC of Canada.}
\thanks{Yun Su, Hans De Sterck, and Jun Liu are with Department of Applied Mathematics, 
        University of Waterloo, 200 University Avenue, Waterloo, ON, Canada N2L 3G1
        {\tt\small yun.su@uwaterloo.ca, hans.desterck@uwaterloo.ca,  j.liu@uwaterloo.ca}}%
}

\begin{document}

\maketitle
\thispagestyle{empty}
\pagestyle{empty}

\begin{abstract}
Leveraging a stochastic extension of Zubov's equation, we develop a physics-informed neural network (PINN) approach for learning a neural Lyapunov function that captures the largest probabilistic region of attraction (ROA) for stochastic systems. We then provide sufficient conditions for the learned neural Lyapunov functions that can be readily verified by satisfiability modulo theories (SMT) solvers, enabling formal verification of both local stability analysis and probabilistic ROA estimates. By solving Zubov's equation for the maximal Lyapunov function, our method provides more accurate and larger probabilistic ROA estimates than traditional sum-of-squares (SOS) methods. Numerical experiments on nonlinear stochastic systems validate the effectiveness of our approach in training and verifying neural Lyapunov functions for probabilistic stability analysis and ROA estimates.
\end{abstract}


\section{Introduction}
Recent advances in neural networks and machine learning have transformed computational research, including the computational aspects of control systems. For deterministic systems, the work \cite{chang2019neural} proposes a learning framework to simultaneously stabilize nonlinear systems with a linear neural controller and learn a neural Lyapunov
function to certify a region of attraction (ROA). This approach was extended to cope with systems of  unknown dynamics in \cite{zhou2022neural}. In many applications, such as power systems \cite{vournas2006region}, robotics \cite{manchester2011regions}, and biological systems \cite{matthews2012region}, accurate estimates of the ROA are essential because they provide critical information for system stability. In \cite{vannelli1985maximal}, a maximal Lyapunov function was introduced that can exactly characterize the ROA. Recent works  \cite{liu2023towards,liu2025physics,kang2023data} have shown that neural network approximations to the solution of Zubov's partial differential equation (PDE) can effectively capture maximal Lyapunov functions, thereby characterizing the entire ROA. Consider the potential approximation error of neural Lyapunov functions, the works \cite{liu2023towards,liu2025physics} also formally verify the satisfaction for Lyapunov condition by satisfiability modulo theories (SMT) solvers.
While these techniques enable tasks such as constructing Lyapunov functions for stability analysis and control design in \cite{chang2019neural, zhou2022neural,  abate2020formal, gaby2022lyapunov, ravanbakhsh2019learning, liu2024tool, dawson2023safe} within deterministic frameworks, many real-world systems are subject to random disturbances. Stochastic models can describe such systems more accurately than deterministic ones. Therefore, we are motivated to investigate probabilistic variants of ROA for stochastic systems.

In this paper, we extend the approach of using neural networks to approximate the solution of Zubov's PDE and the SMT-based verification for neural Lyapunov functions from deterministic systems to stochastic systems. Our goal is twofold: (1) to develop neural Lyapunov functions that approximate probabilistic regions of attraction for stochastic systems, and (2) to formally verify that these neural Lyapunov functions satisfy the stochastic Lyapunov conditions, both locally and over the largest probabilistic region of attraction, i.e., the set of points which are attracted with positive probability to the equilibrium.
\subsection{Related work}
In \cite{zubov1961methods}, V.~I. Zubov introduced a method to characterize the region of attraction using a Lyapunov function that satisfies an appropriate first order partial differential equation, known as Zubov's PDE. In the stochastic setting, the work \cite{camilli2000zubov} extends Zubov's method to stochastic systems, defines the notion of probabilistic ROA, and characterizes probabilistic ROA using the stochastic Zubov PDE. 
Several follow-up works have further explored the characterization of the probabilistic ROA. For example, the work \cite{camilli2002characterizing} builds on the theory in \cite{camilli2000zubov} and extends the framework to compute the almost sure ROA (i.e., attraction with probability one) for stochastic systems. Their Lyapunov function is computed by iterative numerical methods. Moreover, the work \cite{camilli2005zubov} employs numerical methods to compute a Lyapunov function that characterizes the probabilistic ROA for controlled stochastic systems with a prescribed probability. 

In \cite{hafstein2018lyapunov}, although the work focuses on deriving quadratic Lyapunov functions
for autonomous linear stochastic differential equations using sum-of-squares (SOS) methods, it is closely related to our study. This is because an essential step in accurately characterizing the probabilistic ROA is to compute a Lyapunov function for the linearized stochastic system. We extend their work by employing the SOS method for nonlinear stochastic systems, which serves as a comparative experiment later. 

Several works \cite{zhang2022neural, lechner2022stability, ansaripour2022learning} have employed neural networks to learn Lyapunov functions for stochastic systems. The work \cite{zhang2022neural} proposes a learning framework for neural Lyapunov functions and a linear neural stochastic controller. In \cite{lechner2022stability}, neural network–based Lyapunov functions and controllers are learned using reinforcement learning methods, and almost sure asymptotic stability is formally verified for discrete-time nonlinear stochastic control systems. Meanwhile, \cite{ansaripour2022learning} extends the methods of \cite{lechner2022stability, chang2019neural} by employing a learner-verifier framework to jointly learn a control policy and a Lyapunov function. However, to the best of our knowledge, none of these studies analyze the largest probabilistic ROA for stochastic systems. This motivates us to further investigate this area. Using neural networks to solve stochastic Zubov's equation allows us to capture the probabilistic ROA as accurately as possible.  
However, a drawback of neural network approximations is their lack of formal guarantees. To address this issue, SMT solvers such as \cite{gao2013dreal} have been used to formally verify the trained neural Lyapunov functions and provide stability guarantees with verified ROAs in \cite{liu2023towards,liu2025physics}. In this work, we also employ an SMT solver to establish stochastic stability guarantees with formally verified probabilistic ROAs. To enable this, we derive sufficient conditions for the learned neural Lyapunov functions, which is the main contribution of this paper.

\section{Preliminaries}

\subsection{Stochastic Differential Equation and Stochastic Stability}

Consider a probability space \(\left(\Omega, \mathcal{F}, \{\mathcal{F}_t\}_{t \ge t_0}, \mathbb{P}\right)\), where \(\Omega\) is the sample space, \(\mathcal{F}\) is a \(\sigma\)-algebra, \(\{\mathcal{F}_t\}_{t \ge t_0}\) is a filtration satisfying the usual conditions, and \(\mathbb{P}\) is a probability measure. On this space, we consider the stochastic differential equation (SDE):  
\begin{equation}
\label{sde}
\begin{cases}
dX_t = f(X_t)\,dt + \sigma(X_t)\,dB_t, \quad t\ge t_0\\
X_{t_0} = x_0,
\end{cases}
\end{equation}
where \(f:\mathbb{R}^n\to\mathbb{R}^n\) is a locally Lipschitz drift term, \(\sigma:\mathbb{R}^n\to\mathbb{R}^{n\times m}\) is the diffusion coefficient, and 
\[
B_t = \left(B_1(t), \dots, B_m(t)\right)^T
\]
is an \(m\)-dimensional Brownian motion. We denote the solution of \eqref{sde} as $X_t(x_0)$ or $X(t,x_0)$. We refer the readers to \cite{khasminskii2011stochastic} for the standard conditions ensuring existence and uniqueness of solutions to equation (\ref{sde}). In addition, we make the following assumption. 
\begin{ass}
    The functions $f: \mathbb{R}^n \rightarrow \mathbb{R}^n$ and $\sigma: \mathbb{R}^n \rightarrow \mathbb{R}^{n \times m}$ are twice continuously differentiable functions, \( f(0) =  0\), and  \( \sigma(0) = 0 \).
\end{ass}

The following definitions are standard and can be found in \cite[Section 4.2]{mao2007stochastic}. 
\begin{dfn}[Stochastic Stability]
The trivial solution of \eqref{sde} is said to be \emph{stochastically stable} (or stable in probability) if, for every \(\varepsilon \in (0,1)\) and every \(r > 0\), there exists a \(\delta = \delta(\varepsilon, r) > 0\) such that
$$
\mathbb{P}\bigl\{\bigl|X_t(x_0)\bigr| < r \text{ for all } t \ge t_0\bigr\} \ge 1-\varepsilon,
$$
whenever \(\left|x_0\right| < \delta\). Otherwise, the trivial solution is said to be stochastically unstable.
\end{dfn}

\begin{dfn}[Stochastic Asymptotic Stability (SAS)]
{\raggedright
The trivial solution of \eqref{sde} is said to be \emph{stochastically asymptotically stable} if it is stochastically stable and, for every \(\varepsilon \in (0,1)\), there exists a \(\delta_0 = \delta_0(\varepsilon) > 0\) such that
\[
\mathbb{P}\bigl\{\lim_{t \to \infty} X_t(x_0) = 0\bigr\} \ge 1-\varepsilon,
\]
whenever \(\left|x_0\right| < \delta_0\).
\par}
\end{dfn}

\begin{dfn}[SAS in the Large]
The trivial solution of \eqref{sde} is said to be \emph{stochastically asymptotically stable in the large} if it is stochastically stable and,  for all \(x_0 \in \mathbb{R}^n\), 
\[
\mathbb{P}\bigl\{\lim_{t \to \infty} X_t(x_0) = 0\bigr\} = 1.
\]
\end{dfn}


\begin{dfn}[Differential Operator $L$]
    Define the differential operator $L$ associated with equation \eqref{sde} by
    $$
    L:=\sum_{i=1}^n f_i(x) \frac{\partial}{\partial x_i}+\frac{1}{2} \sum_{i, j=1}^n\left[\sigma(x) \sigma^T(x)\right]_{i j} \frac{\partial^2}{\partial x_i \partial x_j} .
    $$
    If $L$ acts on a function $V \in C^{2}\left(\mathcal{D} ; \mathbb{R}\right)$, then
    $$
    L V(x)=V_x(x) f(x)+\frac{1}{2} \operatorname{Tr}[\sigma^T(x) V_{x x}(x) \sigma(x)],
    $$
    where 
    $ V_x=\left(\frac{\partial V}{\partial x_1}, \cdots, \frac{\partial V}{\partial x_n}\right),$ $V_{x x}=\left(\frac{\partial^2 V}{\partial x_i \partial x_j}\right)_{n \times n}=\left(\begin{array}{ccc}
    \frac{\partial^2 V}{\partial x_1 \partial x_1} & \cdots & \frac{\partial^2 V}{\partial x_1 \partial x_n} \\
    \vdots & & \vdots \\
    \frac{\partial^2 V}{\partial x_n \partial x_1} & \cdots & \frac{\partial^2 V}{\partial x_n \partial x_n}
    \end{array}\right)$, and $\operatorname{Tr}[ \cdot ]$ indicates the trace of a matrix.
\end{dfn}
\vspace{0.3cm}
By It\^{o}'s formula, if $X_t \in \mathcal{D}$, then
\begin{equation}
    \label{ito formula}
    d V(X_t)=L V(X_t) d t+V_x(X_t) \sigma(X_t) d B_t.
\end{equation}

\begin{thm}{\cite[Theorem 2.2]{mao2007stochastic}}
    \label{mao thm2.2}
    If there exists a positive-definite function $V\in$ $C^{2}\left(\mathcal{D} ; \mathbb{R}_{+}\right)$ such that $L V(x) \leq 0$ for all $x\in \mathcal {D}$, where $\mathcal D$ is an open set containing the origin, then the trivial solution of  \eqref{sde} is stochastically stable. 
\end{thm}
\begin{thm}{\cite[Theorem 2.3]{mao2007stochastic}}
\label{th:LV<0}
    If there exists a positive-definite function $V \in$ $C^{2}\left(\mathcal{D} ; \mathbb{R}_{+}\right)$ such that $L V(x) <0$ for all $x\in \mathcal{D}\setminus\{0\}$, where $\mathcal D$ is an open set containing the origin, then the trivial solution of \eqref{sde} is stochastically asymptotically stable.
\end{thm}

\subsection{Zubov's Equation and Region of Attraction}
\label{sec:zubov}

\begin{dfn}[Sublevel-$c$ Set]
    Given $V:\mathbb{R}^n\rightarrow \mathbb{R}$, for any $c\in \mathbb{R}$, a set of the form 
    $$V^{c}=\{x \in \mathbb{R}^n \mid V(x) \leq c\} $$ 
    is called a sublevel-$c$ set of $V$.
\end{dfn}
\begin{dfn}[$p$-Region of Attraction]
For any probability level \(p \in [0,1]\), we define the $p$-region of attraction for the origin:
\[
p\text{-}\mathrm{ROA} := \bigl\{x_0 \in \mathbb{R}^{n} : \mathbb{P}\bigl[\lim_{t\to\infty} |X_t(x_0)| = 0\bigr] \geq p\bigr\},
\]
which is the largest probabilistic region of attraction at level \(p\), consisting of all \(x_0\) for which the solution $X_t(x_0)$ converges to the origin with probability at least  \(p\).
\end{dfn}

\begin{dfn}
Denote by $\mathcal{C}$ the set of points that are attracted to the equilibrium with positive probability, i.e.,
$$
\mathcal{C}=\left\{x_0 \in \mathbb{R}^n: \mathbb{P}\bigl[\lim_{t\to\infty} |X_t(x_0)| = 0\bigr] >0\right\}.
$$
\end{dfn}

Similar to how Zubov's equation characterizes the domain of attraction for deterministic systems, Zubov's equation can be used to characterizes probabilistic domains of attraction for stochastic systems \cite{camilli2001generalization,camilli2005zubov,camilli2000zubov}. We present the most relevant part as a preliminary to our work below. 

Introduce a value function $W: \mathbb{R}^n \rightarrow \mathbb{R}$ defined as
\begin{equation}
    \label{eq: W}
    W(x)=1-\mathbb{E}[\exp{(\int_{t_0}^{\infty}- g(X_s(x))ds)}]
\end{equation}
where $g: \mathbb{R}^n \rightarrow \mathbb{R}$ is any Lipschitz continuous function that is positive definite with respect to the origin. Under suitable assumptions \cite[Theorem 2.2]{camilli2002characterizing}, it can be shown that $W$ is the unique solution (in viscosity sense) to the PDE 
\begin{equation}
\label{eq: stochastic zubov}
\begin{cases}
    L W(x) = -g(x)(1 - W(x)),\\ 
    W(0)=0.
\end{cases}
\end{equation}
We refer to this PDE as the stochastic Zubov equation for (\ref{sde}). Furthermore, $\mathcal{C}$ is completely characterized by the strict sublevel-$1$ set of $W$, i.e.,
$$
\mathcal{C}=\{x \in \mathbb{R}^n: W(x)<1\}.
$$
Since \(W\) satisfies $LW(x)=-g(x)(1-W(x))$ and \(g\) is positive definite, \(W(x) < 1\) if and only if \(LW(x) < 0\) when $x\neq 0$. Furthermore, $W(0)=0$ and $LW(0)=0$. Therefore, the sets $\{x \in \mathbb{R}^n : W(x) < 1\}$ and $\{x \in \mathbb{R}^n : LW(x) < 0\}$ are equivalent, provided that $W$ is positive definite. Consequently, by Theorem \ref{th:LV<0}, accurately solving the stochastic Zubov PDE \eqref{eq: stochastic zubov} can lead to a valid Lyapunov function for (\ref{sde}) and an approximation of the probabilistic domain of attraction  $\mathcal{C}$. We will also provide sufficient conditions for verifying $p$-regions of attraction in Section \ref{sec:verify}.

\section{Training Neural Lyapunov Functions via PINNs}

Following \cite{liu2025physics}, we approximate the unique viscosity solution $W$ to the stochastic Zubov equation (\ref{eq: stochastic zubov}) with a neural network \(W_{\mathrm{NN}}(x;\theta)\) through a physics-informed neural network (PINN) framework.

Let \(W_{\mathrm{NN}}(x;\theta)\) denote a neural network approximation for solving the stochastic Zubov PDE \eqref{eq: stochastic zubov}. The loss function consists of three terms:
\[
\mathcal{L}(\theta)=L_{r}(\theta)+L_{b}(\theta)+L_{d}(\theta),
\]
which are defined as follows:
\begin{enumerate}
    \item Residual loss \(L_{r}\): This term measures the error in satisfying the PDE over a set of collocation points \(S=\{x_i\}_{i=1}^{N} \subseteq X\) (with \(X\) being a compact training domain). For example, one may define
\[
\begin{aligned}
L_{r}(\theta) &= \frac{1}{N}\sum_{i=1}^{N}\big(\nabla_{x} W_{\mathrm{NN}}(x_i;\theta)f(x_i) + \\ & \quad\quad  \frac{1}{2}\operatorname{Tr}[\sigma^T(x_i)\,\frac{\partial^2 W_{\mathrm{NN}}}{\partial x^2}(x_i)\,\sigma(x_i)] \\&
\quad\quad + g(x_i)(1-W_{\mathrm{NN}}(x_i;\theta))\big)^2.
\end{aligned}
\]
  \item  Boundary loss \(L_{b}\): This term enforces the boundary conditions. We require that \(W(x)=0\) for \(x =0\). The boundary loss is defined as
  $$L_{b}(\theta)=W_{\mathrm{NN}}(0;\theta)^2.$$
    \item Data loss \(L_{d}\): By an optional set of data points \(S_{d}=\left\{y_{i}\right\}_{i}^{N_{d}}\) on which approximate values \(\hat{W}(y_{i})\) for \(W(y_{i})\) can be obtained. According to the analysis in Section \ref{sec:zubov}, \(W\left(y_{i}\right)\) can be approximated by simulating \(X_t(y_{i})\) for a sufficiently long period of time and use \eqref{eq: W} to compute $\hat{W}\left(y_{i}\right)$. Based on the values of \(\left\{\hat{W}\left(y_{i}\right)\right\}_{i}^{N_{d}}\), the data loss \(L_{d}\) can be defined as
\[
L_{d}=\frac{1}{N_{d}} \sum_{i=1}^{N_{d}}(W_{\mathrm{NN}}\left(y_{i} ; \theta\right)-\hat{W}\left(y_{i}\right))^{2}.
\] 
This term incorporates additional data from stochastic simulations to further improve accuracy.
\end{enumerate}

\section{Formal Verification of the Learned Neural Lyapunov Function}\label{sec:verify}

Due to potential approximation errors and the lack of formal guarantees in neural network approximations, as highlighted in \cite{liu2023towards,liu2025physics,zhou2022neural}, we employ SMT solvers for formal verification. In the following sections, we provide sufficient stochastic Lyapunov conditions for probabilistic attraction, which are the main contribution of this paper. 

\subsection{Verification of Local Stability}

In this section, we assume that Assumption 1 holds. Rewrite equation \eqref{sde} as
\begin{equation}
    \label{eq: linearized sde}
    dX_t = \left( A X_t + m(X_t) \right) dt + \left( C X_t + n(X_t) \right) dB_t,
\end{equation}
where $A = Df(0) \in \mathbb{R}^{n \times n},$
and for $k = 1, \ldots, m$, $S_k = D\sigma_k(0) \in \mathbb{R}^{n \times n},$
with \(\sigma_k(x)\) denoting the \(k\)-th column of \(\sigma(x)\). The matrix \(C\) is given by
$C = [ S_1, S_2, \ldots, S_m ].$
Here, the remaining terms are defined as
$m(x) = f(x) - A x \quad \text{and} \quad n(x) = \sigma(x) - C x,$
which satisfy
$\lim_{x \to 0} \frac{\|m(x)\|}{\|x\|} = 0 \quad \text{and} \quad \lim_{x \to 0} \frac{\|n(x)\|}{\|x\|} = 0.$ Here, $\|\cdot\|$ denotes the matrix 2-norm. When it applies to vectors, $\|\cdot\|$ coincides with vector 2-norm $|\cdot|$ (Euclidean norm).

In \cite[Theorem 3.4]{bjornsson2018local}, local Lyapunov functions for stochastic systems were studied and a corresponding theorem was established. In this paper, we offer a simpler proof under suitable assumptions.
\begin{prop}
\label{prop: LV_p}
Let \(A\) be a Hurwitz matrix. For any symmetric positive definite matrix \(Q \in \mathbb{R}^{n \times n}\), suppose there exists a symmetric positive definite matrix \(P \in \mathbb{R}^{n \times n}\) satisfying the stochastic Lyapunov equation
\begin{equation}
\label{eq:stochastic LF condition}
P A + A^T P + \sum_{i=1}^m S_i^T P S_i = -Q.
\end{equation}
Define
$
V_P(x) = x^T P x.
$
Denote \(LV_P(x)\) by \(h(x)\) and define
\[
M(x) = D^2 h(x) + 2Q,
\]
with \(D^2 h(x)\) being the Hessian matrix of \(h(x)\). For some sufficiently small \(\varepsilon > 0,\) set 
$
r = \lambda_{\min}(Q) - \varepsilon > 0.
$
If there exists a constant \(c>0\) such that
\begin{equation}
\label{local condition}
x\in V_P^c = \{ x \in \mathbb{R}^n \mid V_P(x) \le c \} \Rightarrow \| M(x) \|_F \le 2 r,
\end{equation}
then
\[
LV_P(x) < 0, \quad \forall x \in V_P^c\setminus\{0\}.
\]
\end{prop}

\begin{proof}
We begin by computing the generator \(LV_P(x)\) for the function \(V_P(x)=x^T P x\). By definition, using It\^{o}'s formula, we have
\begin{equation}
\label{trace term}
\begin{aligned}
LV_P(x) =& \frac{dV_P}{dx} f(x) + \frac{1}{2} \operatorname{Tr}[ \sigma(x)^T V_{xx}(x) \sigma(x)]\\
=& 2x^T P A x + 2x^T P m(x) + \\ & \operatorname{Tr}[(Cx+n(x))^T P (Cx+n(x))].
\end{aligned}
\end{equation}
Expanding the trace term\footnote{We have \( C^T P C = [(S_i x)^T P(S_j x)]_{i,j=1}^m,\)    \(\operatorname{Tr}[C^T P C] = \sum_{i=1}^m x^T S_i^T P S_i x. \)} in \eqref{trace term}
gives
\[
 \sum_{i=1}^m x^T S_i^T P S_i x + 2\operatorname{Tr}[n(x)^T P C x] + \operatorname{Tr}[n(x)^T P n(x)].
\]
Thus, we can write
\[
\begin{aligned}
LV_P(x) =\, & x^T(PA + A^T P + \sum_{i=1}^m  S_i^T P S_i )x + 2 x^T P m(x)\\
& + \operatorname{Tr}[2 n(x)^T P C x + n(x)^T P n(x)].
\end{aligned}
\]
By \eqref{eq:stochastic LF condition}, it follows that
\[
\begin{aligned}
LV_P(x) &= -x^T Q x + 2x^T P m(x) + \operatorname{Tr}[2 n(x)^T P C x \\
         &\quad + n(x)^T P n(x)].
\end{aligned}
\]
Since \(\sigma(0)=n(0)=0\), we have \(LV_P(0)=0\). Moreover, since both \(m(x)\) and \(n(x)\) vanish faster than \(x\) as \(x\) approaches zero, the dominant term in the expansion of \(LV_P(x)\) is quadratic with a coefficient of \(-Q\).
Define \(h(x)=LV_P(x)\). Then \(h(0)=0\), \(\nabla h(0)=0\) and the Hessian at \(x=0\) satisfies
$D^2 h(0) = -2Q.$

By Taylor's theorem in several variables, we can express
\[
\begin{aligned}
h(x) &= h(0) + \nabla h(0)\cdot x + x^T[\int_0^1 D^2h(tx)\,(1-t)\,dt]x\\
     &= -x^T Q x + R(x),
\end{aligned}
\]
where the remainder \(R(x)\) can be written as
\[
R(x) = x^T [\int_0^1 ( D^2 h(tx) + 2Q )(1-t)\,dt] x.
\]
Define
$$
M(x) = D^2 h(x) + 2Q.
$$
Then,
\[
h(x) = -x^T Q x + x^T [\int_0^1 M(tx)(1-t)\,dt] x.
\]
Assume that, for all \(x \in V_P^c\), \eqref{local condition} holds.
Then, by estimating the remainder term, we have
\[
\begin{aligned}
h(x) &\le -\lambda_{\min}(Q)\|x\|^2 + \|x\|^2 \int_0^1 \| M(tx) \|_F (1-t) dt\\
&\le -\lambda_{\min}(Q)\|x\|^2 + 2r \|x\|^2 \int_0^1 (1-t) dt\\
&= (r - \lambda_{\min}(Q))\|x\|^2\\
& \leq -\varepsilon \|x\|^2,  \quad \quad \quad \quad \quad \quad \quad  \quad \quad \quad \forall x \in V_P^c.
\end{aligned}
\]
Thus, \(LV_P(x) = h(x) < 0\) for all \(x \in V_P^c \setminus\{0\}\).
\end{proof}

\begin{rem}
For any matrix \(A \in \mathbb{R}^{n \times n}\), we have
\[
\|A\| \le \|A\|_F \le \sqrt{n}\|A\|,
\]
which shows that the 2-norm and Frobenius norm are equivalent. In practice, we use the Frobenius norm due to its computational simplicity, particularly in the context of formal verification with SMT solvers. 

Since \(V_P^c\) is a convex set containing 0, the implication
\[
x^T P x \le c \implies \|M(tx)\|_F \le 2r,\quad \forall\, t\in [0,1],
\]
is equivalent to (\ref{local condition}). 
For any given $r>0$, condition \eqref{local condition} can be verified by dReal \cite{gao2013dreal}. Since \(\|M(0)\|_F = 0< 2r\) and $P$ is positive definite, one can always choose \(c>0\) sufficiently small so that \eqref{local condition} holds. In rare cases, if \eqref{local condition} holds for all \(x\in\mathbb{R}^n\), then the solution of \eqref{sde} is stochastically asymptotically stable in the large.  
\end{rem}

\subsection{Verification of Probabilistic Regions of Attraction}

In \cite[Theorem 4.4]{meng2024stochastic}, reach-avoid-stay and safe-stability for stochastic systems are established using a single Lyapunov function.  Here, we reformulate the probabilistic safe stability by employing the concept of sublevel sets. This reformulation provides a precise definition of probabilistic attraction and yields a straightforward, verifiable condition for ensuring probabilistic safe stability.

\begin{thm}
\label{th: LV}
    Let Assumption $1$ hold and $\mathcal{D}\subseteq \mathbb{R}^n$. Consider a function $V \in C^2\left(\mathbb{R}^n;\mathbb{R}_{\geq 0}\right)$. Let $c_2>0$ be such that $V^{c_2}\subset \mathcal{D}$. The following holds: 
    \begin{enumerate}[label=\roman*)]
        \item  If $L V(x) \leq 0$ for all $ x \in V^{c_2}$, then for any $x_0 \in V^{c_2}$, we have
        $$
        \mathbb{P}\left\{X_t(x_0) \in V^{c_2}, \,\, \forall t \geq t_0\right\} \geq 1 - \frac{V(x_0)}{c_2}.
        $$
        
        
        
        
        \item 
        If $L V(x) < 0$ for all $ x \in V^{c_2} \setminus \{0\}$  and $V$ is positive definite on $\mathcal D$, then, for any $x_0 \in V^{c_2}$, we have
        {\small
        $$
        \mathbb{P}\{\lim _{t \rightarrow \infty} X_t(x_0) = 0 ,\, X_t(x_0) \in V^{c_2}, \forall t \geq t_0\} \geq 1-\frac{V(x_0)}{c_2}.
        $$}

        \item  If there exists $0<c_1<c_2$ and $\zeta>0$ such that \( LV(x) \le -\zeta \) for all \( x \in V^{c_2} \setminus V^{c_1} \).
        Then, for any \( x_0 \in V^{c_2} \setminus V^{c_1}  \), 
        {\small\[
        \begin{aligned}
        &\mathbb{P}\{X_t(x_0) \in V^{c_2}, \, \forall t \geq t_0 \text{ and } \exists \, T\ge t_0\, 
        \text{s.t.}\,  X_T(x_0) \in V^{c_1}\} \\
        &\geq 1 - \frac{V(x_0)}{c_2}.
        \end{aligned}
        \]}
    \end{enumerate}
        
\end{thm}

\begin{proof}
 \romannumeral 1)
For any initial condition $x_0 \in V^{c_2}$, define $\tau$ be the first exit time of $X_t(x_0)$ from $V^{c_2}$, that is 
$$
\tau=\inf \left\{t \geq t_0 : X_t(x_0) \notin V^{c_2} \right\}.
$$ 
If $X_t(x_0) \in V^{c_2}$ for all $t \geq t_0$, then $\tau=\inf \emptyset=\infty$. 
In the following proof, we simplify the notation by writing \(X_s(x_0)\) as \(X_s\).
By It\^{o}'s formula, 
\[
\begin{aligned}
V(X_{\tau \wedge t}(x_0)) &= V(x_0) + \int_{t_0}^{\tau \wedge t} L V(X_s) \, ds \\
&+ \int_{t_0}^{\tau \wedge t} V_x(X_s) g(X_s) \, dB(s), \quad \forall t \ge t_0.
\end{aligned}
\]
For any fixed $t \geq t_0$, taking expectation on both sides and making use of $L V(X_t(x_0)) \leq 0$ for all $ X_t(x_0) \in V^{c_2} \subseteq \mathcal{D}$, we can get
\begin{equation}
    \label{eq1: p-roa}
    \mathbb{E} V(X_{\tau \wedge t}(x_0))=V\left(x_0\right)+\mathbb{E} \int_{t_0}^{\tau \wedge t} L V(X_s) d s \leq V\left(x_0\right).
\end{equation}
Considering $\mathbb{E} V(X_{\tau \wedge t}) = \mathbb{E}[(I_{\{\tau \leq t\}} +I_{\{\tau > t\}})V(X_{\tau \wedge t})]$ and $V(X_t) \geq 0$ for all $X_t \in \mathcal{D}$, we can get
\begin{equation}
    \label{eq2: p-roa}
    \mathbb{E}[I_{\{\tau \leq t\}} V(X_{\tau \wedge t})] \leq \mathbb{E} V(X_{\tau \wedge t}).
\end{equation}
By the definitions of $V^{c_2}$ and $\tau$ and continuity of $V(x)$ and $X_t$, if $\tau \leq t$, we know that $V(X_{\tau \wedge t}(x_0)) = V(X_{\tau}(x_0)) = c_2$. Then we can get 
\begin{equation}
    \label{eq3: p-roa}
    \mathbb{P}\left\{\tau \leq t\right\} c_2 = \mathbb{E}[I_{\{\tau \leq t\}} V(X_{\tau \wedge t}(x_0))].
\end{equation}
By \eqref{eq1: p-roa}, \eqref{eq2: p-roa} and \eqref{eq3: p-roa}, we can get
$$ \mathbb{P}\left\{\tau \leq t\right\} c_2= \mathbb{E}[I_{\{\tau \leq t\}} V(X_{\tau \wedge t})] \leq \mathbb{E} V(X_{\tau \wedge t}) \leq V\left(x_0\right). $$
Therefore, $ \mathbb{P}\{\tau \leq t\} \leq \frac{V\left(x_0\right)}{c_2}.$ 
By letting $t \rightarrow \infty$, we obtain $\mathbb{P}\{\tau<\infty\} \leq \frac{V\left(x_0\right)}{c_2}$. It follows that, for all $x_0 \in V^{c_2}$, 
\begin{equation}
\label{eq4: p-roa}
\mathbb{P}\{\tau=\infty \} 
= 1-\mathbb{P}\{\tau<\infty\} 
\geq  1-\frac{V\left(x_0\right)}{c_2}.
\end{equation}
Since $\tau = \infty$ is equivalent to $X_t(x_0) \in V^{c_2}$ for all $t \geq t_0$. We can rewrite \eqref{eq4: p-roa} as
\begin{equation}
\label{eq5: p-roa}
\mathbb{P}\left\{X_t(x_0) \in V^{c_2}, \forall t \geq t_0\right\} \geq 1 - \frac{V\left(x_0\right)}{c_2}.
\end{equation}

\romannumeral 2)
Fix an arbitrary $r \geq 0$ such that $\mathcal{B}_{r}(0) \subseteq V^{c_2}$. Note that $V$ satisfies the conditions in Theorem \ref{mao thm2.2}. By the definition of stochastic stability, it follows that there exists $\eta \in\left(0, r\right)$ such that for any $\varepsilon \in(0,1)$, $x(\tau^*,x_0) \in \overline{\mathcal{B}}_\eta(0)$ implies
$\mathbb{P}\{X_t(x_0) \in \mathcal{B}_{r}(0), \, \forall t \geq \tau^*\} \geq 1-\varepsilon$. 
Next, we define $\tau^*$ to be the first entry time of $X_t(x_0)$ to $\mathcal{B}_\eta(0)$, that is 
$$
\tau^*=\inf \left\{t \geq t_0 : X_t(x_0) \in \mathcal{B}_\eta( 0 )\right\}.
$$ 
    
Let $\xi=- \max_{x\in V^{c_2} \setminus \mathcal{B}_\eta(0)} LV(x)$. By It\^{o}'s formula and $LV < -\xi$ on $V^{c_2} \setminus \mathcal{B}_\eta(0)$, we have
\begin{equation}
\label{eq6:p-roa}
\begin{aligned}
0 & \leq V(X_{\tau^* \wedge \tau \wedge t}) = V(x_0)+{\mathbb{E}} \int_{t_0}^{\tau^* \wedge \tau \wedge t} L V(X_s) d s \\
& \leq V(x_0)-\xi {\mathbb{E}}\left[\tau^* \wedge \tau \wedge t-t_0\right].
\end{aligned}
\end{equation}
Since on $\left\{\tau^* \wedge \tau \geq t\right\}$ we have $\tau^* \wedge \tau \wedge t=t$,
\[
\begin{aligned}
\mathbb{E}\left[\tau^* \wedge \tau \wedge t - t_0\right] &\geq \int_{\Omega} 1_{\{\tau^* \wedge \tau \geq t\}} \cdot (t-t_0)\,d\mathbb{P}(\omega) \\
&= (t-t_0)\, \mathbb{P}\left[\tau^* \wedge \tau \geq t\right].
\end{aligned}
\]
Combining this with \eqref{eq6:p-roa}, we have
$$
\mathbb{P}\left[\tau^* \wedge \tau \geq t\right] \leq \frac{V(x_0)}{(t-t_0) \xi} , \text { for each } t.
$$
By letting $t \rightarrow \infty$, we get $\mathbb{P}\left[\tau^* \wedge \tau \geq \infty\right] = 0$, which implies $\mathbb{P}\left[\tau^* \wedge \tau<\infty\right]=1$. On $\{\tau=\infty\}$, we have $\mathbb{P}\left[\tau^*<\infty\right]=1$ and
$$
\begin{aligned}
& \mathbb{P}\left[\underset{t \rightarrow \infty}{\limsup} \, |X_t(x_0)| \leq r\right] \\
= & \mathbb{P}\left[\underset{t \rightarrow \infty}{\limsup} \, |X_t(x_0)| \leq r \mid 
 \tau^*<\infty \right] \\
\geq & \mathbb{P}\left[ |X_t(x_0)| \leq r, \forall t \geq \tau^* \mid \tau^*<\infty\right] \geq 1-\varepsilon .
\end{aligned}
$$
Since $\varepsilon$ and $r$ are arbitrary, we can obtain
\begin{equation}
\label{eq7: p-roa}
    \mathbb{P}\{\lim _{t \rightarrow \infty}|X_t(x_0)|=0 \mid \tau=\infty\}=1.
\end{equation}
Since $\tau = \infty$ is equivalent to $X_t(x_0) \in V^{c_2}$ for all $t \geq t_0$. We can rewrite \eqref{eq6:p-roa} as
$$
\mathbb{P}\{\lim _{t \rightarrow \infty}|X_t(x_0)|=0 \mid X_t(x_0) \in V^{c_2}\}=1 .
$$
Furthermore, combining this with \eqref{eq5: p-roa}, we can obtain 
$$
\begin{aligned}
&  \mathbb{P}\{\lim _{t \rightarrow \infty} X_t(x_0) = 0 \,\, \text{and} \,\, X_t(x_0) \in V^{c_2}, \, \forall t \geq t_0\} \\
= & \mathbb{P}\{\lim _{t \rightarrow \infty}|X_t(x_0)|=0 \mid X_t(x_0) \in V^{c_2}, \forall t \geq t_0\} \\
&  \times\, \mathbb{P}\left\{X_t(x_0) \in V^{c_2}, \,  \forall t \geq t_0\right\}\\
\geq & 1- \frac{V(x_0)}{c_2}.
\end{aligned}
$$
\romannumeral 3)
The proof follows the same strategy as in \romannumeral 2). 
For any initial condition $x_0 \in V^{c_2}\setminus V^{c_1}$, let $\tau$ be as defined in the proof of part i) and \( \sigma \) be the first entry time of $X_t(x_0)$ to \( V^{c_1} \), given by
\[
\sigma = \inf \left\{t \geq t_0 : X_t(x_0) \in V^{c_1} \right\}.
\]
For \( t \le \tau \wedge \sigma \), \( X_t(x_0) \in V^{c_2} \setminus V^{c_1} \), where \( LV(X_t(x_0)) \le -\zeta \). We have
\begin{equation}
\begin{aligned}
\label{eq: zeta}
0 &\leq \mathbb{E}[V(X_{\tau \wedge \sigma \wedge t})] = V(x_0) +{\mathbb{E}} \int_{t_0}^{\sigma \wedge \tau \wedge t} L V(X_s) d s\\
& \le V(x_0)-\zeta {\mathbb{E}}\left[\sigma \wedge \tau \wedge t-t_0\right].
\end{aligned}
\end{equation}
Since on $\left\{\sigma \wedge \tau \geq t\right\}$ we have $\sigma \wedge \tau \wedge t=t$,
\[
\begin{aligned}
\mathbb{E}\left[\sigma \wedge \tau \wedge t - t_0\right] &\geq \int_{\Omega} 1_{\{\sigma \wedge \tau \geq t\}} \cdot (t-t_0)\,d\mathbb{P}(\omega) \\
&= (t-t_0)\, \mathbb{P}\left[\sigma \wedge \tau \geq t\right].
\end{aligned}
\]
Combining this with \eqref{eq: zeta}, we have
$$
\mathbb{P}\left[\sigma \wedge \tau \geq t\right] \leq \frac{V(x_0)}{(t-t_0) \zeta} , \text { for each } t.
$$
By letting $t \rightarrow \infty$, we get $\mathbb{P}\left[\sigma \wedge \tau \geq \infty\right] = 0$, which implies $\mathbb{P}\left[\sigma \wedge \tau<\infty\right]=1$. On $\{\tau=\infty\}$, we have $\mathbb{P}\left[\sigma<\infty\right]=1$,
that is,
\begin{equation}
\label{eq8: p-roa}
    \mathbb{P}\{\exists T\ge t_0\,\text{ s.t. }\, X_T(x_0) \in V^{c_1} \mid \tau=\infty\}=1.
\end{equation}
Since $\tau = \infty$ is equivalent to $X_t(x_0) \in V^{c_2}$ for all $t \geq t_0$, we can rewrite \eqref{eq8: p-roa} as
$$
\mathbb{P}\{\exists T\ge t_0\,\text{ s.t.}\, X_T(x_0) \in V^{c_1} \mid X_t(x_0) \in V^{c_2}\}=1 .
$$
With a slightly modified argument as in the proof of part ii), we can show that
\[
\mathbb{P}\{ X_t(x_0) \in V^{c_2},\,\forall t\ge t_0\} \geq 1-\frac{V(x_0)}{c_2}.
\]
Thus,
{
\[
    \begin{aligned}
    &\mathbb{P}\{X_t(x_0) \in V^{c_2}, \, \forall t \geq t_0 \, \text{and}\,\exists T\ge t_0\,\text{s.t.}\,  X_T(x_0) \in V^{c_1} \} \\
    &\geq 1 - \frac{V(x_0)}{c_2}, \quad \quad \quad \forall x \in V^{c_2}\setminus V^{c_1} .
    \end{aligned}
\]}
\end{proof}

In our computation, both quadratic Lyapunov functions and neural Lyapunov functions are used to certify probabilistic attraction. For this purpose, we need to formulate the following result. Given one Lyapunov function, we seek to use another to enlarge the probabilistic region of attraction, as stated in the following theorem.

\begin{thm}
\label{th: enlarge p-roa}
Let Assumption $1$ hold and $\mathcal{D}\subseteq \mathbb{R}^n$. Let $V\in C^2\left(\mathcal{D};\mathbb{R}_{\geq 0}\right)$ satisfy the conditions in Theorem \ref{th: LV} iii) for some $c_2>0$. Let $c_1\in (0,c_2)$. 
Assume that there exists a function $ W\in C^2\left(\mathcal{D};\mathbb{R}_{\geq 0}\right)$ and some $\zeta>0$ satisfying 
\begin{equation}\label{eq:LW}
L W(x) \le -\zeta, \quad \quad \forall x \in W^{\beta_2} \setminus W^{\beta_1},   
\end{equation}
where $\beta_2>\beta_1>0$ and $ W^{\beta_2} \subseteq \mathcal{D}$. 
Furthermore, assume that $W^{\beta_1}\subseteq V^{c_1}\subseteq V^{c_2} \subseteq W^{\beta_2} \subseteq \mathcal{D}$. Then, for any \( x_0 \in  W^{\beta_2} \), we have 
\[
\begin{aligned}
\mathbb{P}\{ &\lim_{t \to \infty} X_t(x_0) = 0 \text{ and } X_t(x_0) \in W^{\beta_2},\ \forall t \geq t_0 \} \\
&\geq ( 1 - \dfrac{c_1}{c_2} ) ( 1 - \frac{W(x_0)}{\beta_2} ).
\end{aligned}
\]
That is, for all \( x_0 \in   W^{\beta_2}\), \( X_t(x_0) \) will remain in $W^{\beta_2}$ and converge to the equilibrium point with probability at least \( ( 1 - \dfrac{c_1}{c_2} ) ( 1 - \dfrac{W(x_0)}{\beta_2} ) \).
\end{thm}
\begin{proof}
For any initial condition $x_0 \in W^{\beta_2} \setminus  W^{\beta_1}$,
denote by $E_1$ the following set:
\[
\left\{X_t(x_0) \in W^{\beta_2}, \, \forall t \geq t_0\, \text{and }\exists \, T\ge t_0\text{ s.t.}\, X_T(x_0) \in V^{c_1}\right\} 
\]
It follows from \eqref{eq:LW}, Theorem \ref{th: LV} iii) (applied to $W$), and $ W^{\beta_1} \subseteq V^{c_1}$ that, for any \( x_0 \in W^{\beta_2} \setminus  W^{\beta_1}\),
\[
\mathbb{P}(E_1)\geq 1-\frac{W(x_0)}{\beta_2}.
\]

Define $\tau$ be the first time $X_t(x_0)$ enters the set $V^{c_1}$. The state $X_{\tau}(x_0)$ serves as the initial condition for the evolution of $X_t(X_{\tau}(x_0))$. 

Denote by $E_2$ the following set:
\[
\bigl\{ \lim_{t \to \infty} X_t(X_{\tau}(x_0)) = 0 \text{ and } X_t(X_{\tau}(x_0)) \in V^{c_2}, \, \forall t \geq 0 \bigr\}.
\]
By the assumptions on $V$ and Theorem \ref{th: LV} ii), we obtain
\[
\mathbb{P}(E_2\mid E_1)\geq 1-\frac{V(X_{\tau}(x_0))}{c_2} = 1-\frac{c_1}{c_2}.
\]
It follows that
\[
\mathbb{P}(E_1 \cap E_2)\\
=  \mathbb{P}(E_2 \mid E_1)\mathbb{P}(E_1)\\
\geq ( 1 - \frac{c_1}{c_2} ) ( 1 - \frac{W(x_0)}{\beta_2} ).
\]
\end{proof}

Taking the computed quadratic Lyapunov function as \(V\) and the neural Lyapunov function as \(W\), by Theorem \ref{th: enlarge p-roa}, we obtain verified probabilistic ROAs for the neural Lyapunov function. This approach effectively resolves the issue that a neural Lyapunov function may not be well-trained and fail to satisfy the Lyapunov condition around the origin. We derive a general attraction probability function using a piecewise formulation that characterizes a conservative lower bound on the attraction probability for an initial point in \(\mathcal{C}\), as stated in the following proposition.

\begin{prop}
\label{prop:p-roa}
Let $V$ and $W$ be the functions and  \(\beta_1\), \(\beta_2\), \(c_1\), and \(c_2\) be the constants defined in Theorem \ref{th: enlarge p-roa}. Then, the attraction probability for an initial state \(x_0 \in W^{\beta_2}\), $\mathbb{P}(x_0)$, defined as
$$
\mathbb{P}(x_0) := \mathbb{P}\bigl\{ \lim_{t \to \infty} X_t(x_0) = 0\bigr\},
$$
is at least
$$
p(x_0):= {
\begin{cases}
1 - \dfrac{V(x_0)}{c_2}, \quad \quad \quad \quad \quad \quad \quad \text{if } W(x_0) < \beta_1, \\
\displaystyle \max\bigl\{  \bigl(1 - \frac{W(x_0)}{\beta_2}\bigr)
\bigl(1 - \frac{c_1}{c_2}\bigr), \\
\quad 1 - \dfrac{V(x_0)}{c_2}\bigr\}, \quad \quad \quad \, \text{if } \beta_1 \le W(x_0) \le \beta_2.
\end{cases}}
$$
Furthermore, \(p(x_0)\) is continuous with respect to \(x_0\) and serves as a conservative lower bound on the attraction probability.
\end{prop}

\begin{proof}
The fact that \(\mathbb{P}(x_0) \ge p(x_0)\) follows immediately from  Theorems \ref{th: LV} and \ref{th: enlarge p-roa}.  Indeed, for any $x_0$, we have $\mathbb{P}(x_0)\ge 1-V(x_0)/c_2$ by Theorem \ref{th: LV} iii) and $\mathbb{P}(x_0)\ge (1 - W(x_0)/\beta_2)
(1 - c_1/c_2)$ by Theorem \ref{th: enlarge p-roa}. Note that these bounds hold even if $V(x_0)>c_2$ or $W(x_0)>\beta_2$ (although then the bounds will be trivial). 

We now show that $p(x_0)$ is continuous in $x_0$. We only need to verify this at every $x_0$ such that $W(x_0)=\beta_1$. Since $W^{\beta_1}\subset V^{c_1}$, we have $V(x_0)\le c_1$. Hence, $1-V(x_0)/c_2\ge 1-c_1/c_2 > (1 - W(x_0)/\beta_2)
(1 - c_1/c_2)$. By continuity of both \(V(x)\) and \(W(x)\) on \(\mathcal{D}\), we have $1-V(x_0)>(1 - W(x_0)/\beta_2)
(1 - c_1/c_2)$ holds in a small neighborhood of $x_0$, on which we have $p(x_0)=1-V(x_0)/c_2$. Hence, $p(x_0)$ is continuous at $x_0$ because $V$ is. 
\end{proof}

\begin{rem}
To apply Proposition \ref{prop:p-roa}, we need to determine the constants \(\beta_1\), \(\beta_2\), \(c_1\), and \(c_2\) as described in Theorem \ref{th: enlarge p-roa}. We can choose them as follows to optimize the lower bound \( p(x_0) \) given by Proposition \ref{prop:p-roa}. Suppose we already have a Lyapunov function \( V \) such that \( LV(x) < 0 \) for all \( x \in V^{c_2} \setminus \{0\} \), where \( c_2 \) is chosen as the largest such value.  A premise for applying Theorem \ref{th: enlarge p-roa} to further enlarge the probabilistic region of attraction is the existence of another Lyapunov function \( W \) such that \( LW(x) < 0 \) holds on \( W^{\beta_2} \setminus W^{\beta_1} \), where \( W^{\beta_1} \subseteq V^{c_1} \subseteq V^{c_2} \subseteq W^{\beta_2} \). We need \( V^{c_2} \subseteq W^{\beta_2} \) so that the region of attraction can be enlarged using the Lyapunov function \( W \), and we need \( W^{\beta_1} \subseteq V^{c_1} \) to draw conclusions about probabilistic attraction using the Lyapunov function \( V \).  Hence, \( \beta_1 \) is chosen as the smallest and \( \beta_2 \) as the largest such that \( LW(x) < 0 \) holds on \( W^{\beta_2} \setminus W^{\beta_1} \). The constant \( c_1 \) is the smallest value ensuring that \( \{x : W(x) \le \beta_1\} \subseteq \{x : V(x) \le c_1\} \). Clearly, minimizing \( \beta_1 \) and \( c_1 \), and maximizing \( \beta_2 \), leads to less conservative results in terms of \( p \)-ROA estimates.
\end{rem}

\section{Numerical Examples}

In this section, we present numerical experiments on stochastic systems. To demonstrate that the proposed PINN approach reliably approximates \(W(x)\), we extend the LyZNet tool \cite{liu2024tool} to learn a neural Lyapunov function and use dReal \cite{gao2013dreal} to formally verify the stochastic Lyapunov condition. This enables the characterization of the largest probabilistic region of attraction \(\mathcal{C}\). The results confirm that the learned functions capture the probabilistic attraction accurately and quantify the attraction probabilities via Proposition \ref{prop:p-roa}.

\begin{eg}[Reversed Van der Pol]

\label{van_der_pol}
Consider the reversed Van der Pol oscillator
$$
 \begin{cases}
d{X}_{1}(t)=-X_{2}(t) dt +\alpha  X_1(t) dB_1(t), \\
d{X}_{2}(t)=(X_{1}(t) - (1-X_{1}^{2}(t) ) X_{2}(t)) dt + \beta  X_2(t) dB_2(t),
\end{cases}
$$
which has a stable equilibrium point at the origin. The linearization at the origin is given by \eqref{eq: linearized sde} with \(A=\left[\begin{array}{ll}0 & -1 \\ 1 & -1\end{array}\right]\), $S_1=\left[\begin{array}{ll}\alpha & 0 \\ 0 & 0 \end{array}\right]$ and $S_2=\left[\begin{array}{ll}0 & 0 \\ 0 & \beta \end{array}\right]$. Taking the diffusion parameters as $\alpha = \beta=0.5$ and solving the stochastic Lyapunov equation \eqref{eq:stochastic LF condition} with \(Q=I\) gives \(P=\left[\begin{array}{cc}2.2439 & -0.7805 \\ -0.7805 & 1.4634\end{array}\right]\). With \(r=0.9999\), we can verify \eqref{local condition} with \(c=0.3317\). Next, we preform an additional step to check that $ LV_p\le -\epsilon$ holds on the ellipsoid \(\Omega_{c}=\left\{x \in \mathbb{R}^{n}: x^{T} P x \leq c\right\}\) with $c=2.3151$, for some small $\epsilon>0$. Hence, $\Omega_c$ provides a verified local probabilistic ROA for the origin.

In order to better capture the largest probabilistic ROA \(\mathcal{C}\), we employ a neural network with three hidden layers (10 neurons per layer) trained using Zubov's PDE \eqref{eq: stochastic zubov}.  For data generation, we estimate the expectation in \eqref{eq: W} by running 100 simulations for each initial point. In this work, we set \(g(x)=0.1 \|x\|^{2}\), same as the setting for the deterministic system in \cite{liu2025physics}. The learned neural Lyapunov functions are formally verified using dReal. Additionally, we compare the verified probabilistic ROA \(\mathcal{C}\) from our neural Lyapunov function with the Lyapunov function obtained via the SOS method.   

\begin{figure}[h]
    \centering
    \includegraphics[width=0.9\linewidth]{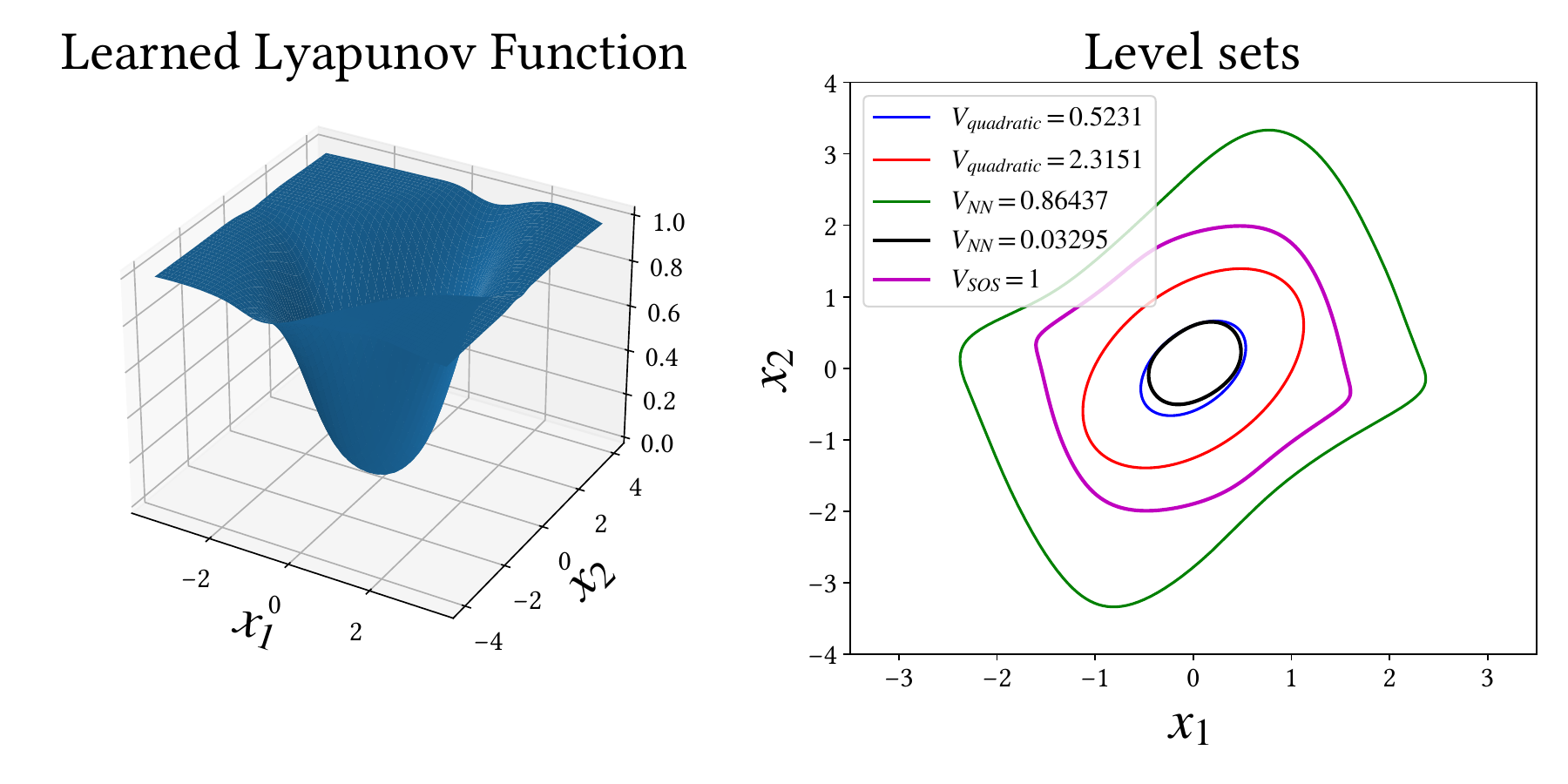}
    \caption{A Neural Lyapunov function trained using stochastic Zubov PDE \eqref{eq: stochastic zubov} is capable of approximately representing a verifiable Lyapunov function that closely approximates the probabilistic ROA for the stable equilibrium point of the Van der Pol equation. The green, pink, and red curves represent the verified \(\mathcal{C}\) for the neural Lyapunov function, the SOS-based Lyapunov function, and the quadratic Lyapunov function, respectively. The black curve represents the region where the neural Lyapunov function is not well-trained and fails to satisfy the stochastic Lyapunov condition. The blue curve indicates a lower bound for the quadratic Lyapunov function that encloses the black curve. With the verified level sets, the attraction probability function can be computed by Proposition \ref{prop:p-roa}.}
    \label{fig: nn_LF and level sets}
\end{figure}

The results are shown in Figure \ref{fig: nn_LF and level sets}. In the right panel of Figure \ref{fig: nn_LF and level sets}, the green curve represents the largest verified probabilistic ROA in terms of a sublevel set of the neural Lyapunov function on which the stochastic Lyapunov condition is satisfied. The pink curve shows the verified sublevel set computed via an SOS Lyapunov function, and the red curve corresponds to that obtained from a quadratic Lyapunov function. Note that the neural Lyapunov function yields a larger verified \(\mathcal{C}\) than the other two approaches.
Additionally, the black curve marks the lower bound of the region where the neural Lyapunov function satisfies the stochastic Lyapunov condition, while the blue curve indicates a lower bound for the quadratic Lyapunov function that encloses the region where the neural Lyapunov function fails to verify the condition.

\begin{figure}[h]
    \centering
    \includegraphics[width=0.46\linewidth]{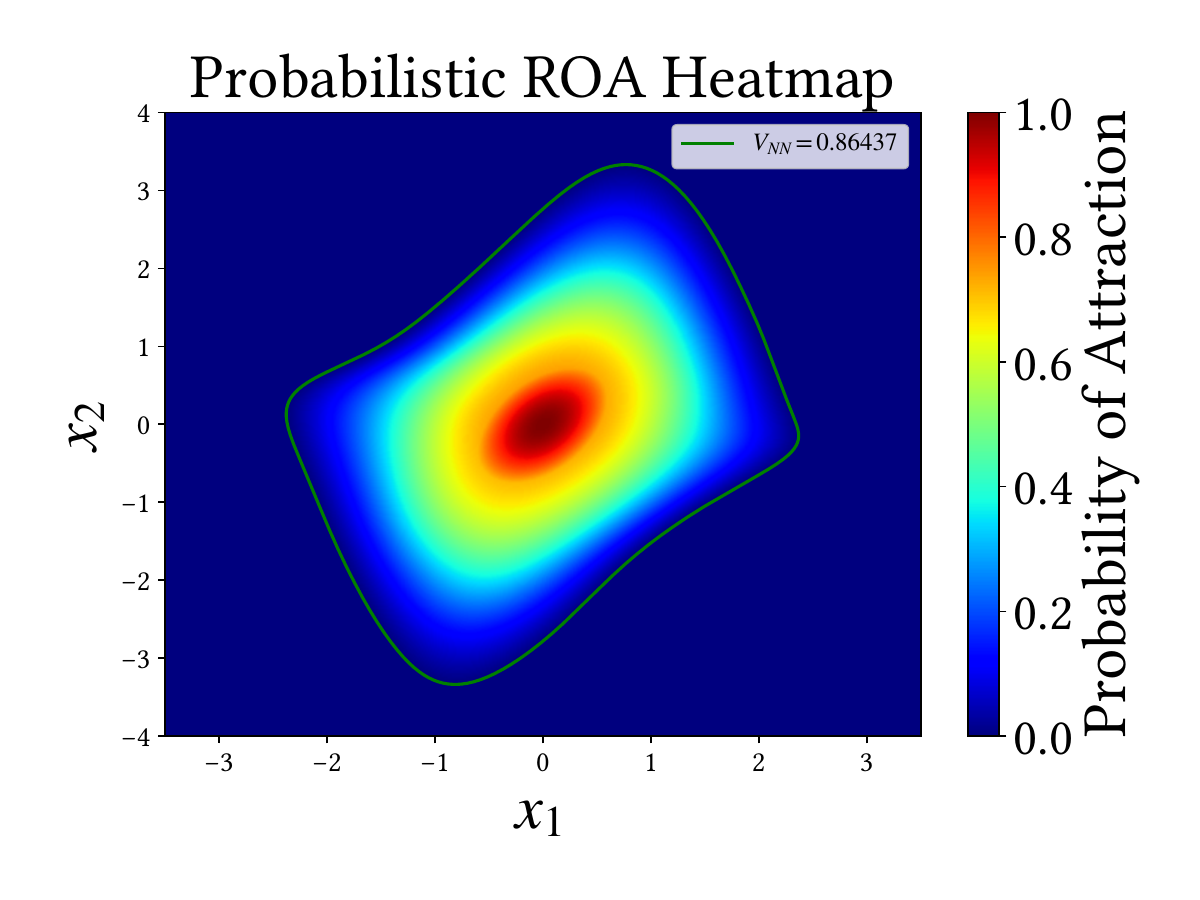}
    \includegraphics[width=0.45\linewidth]{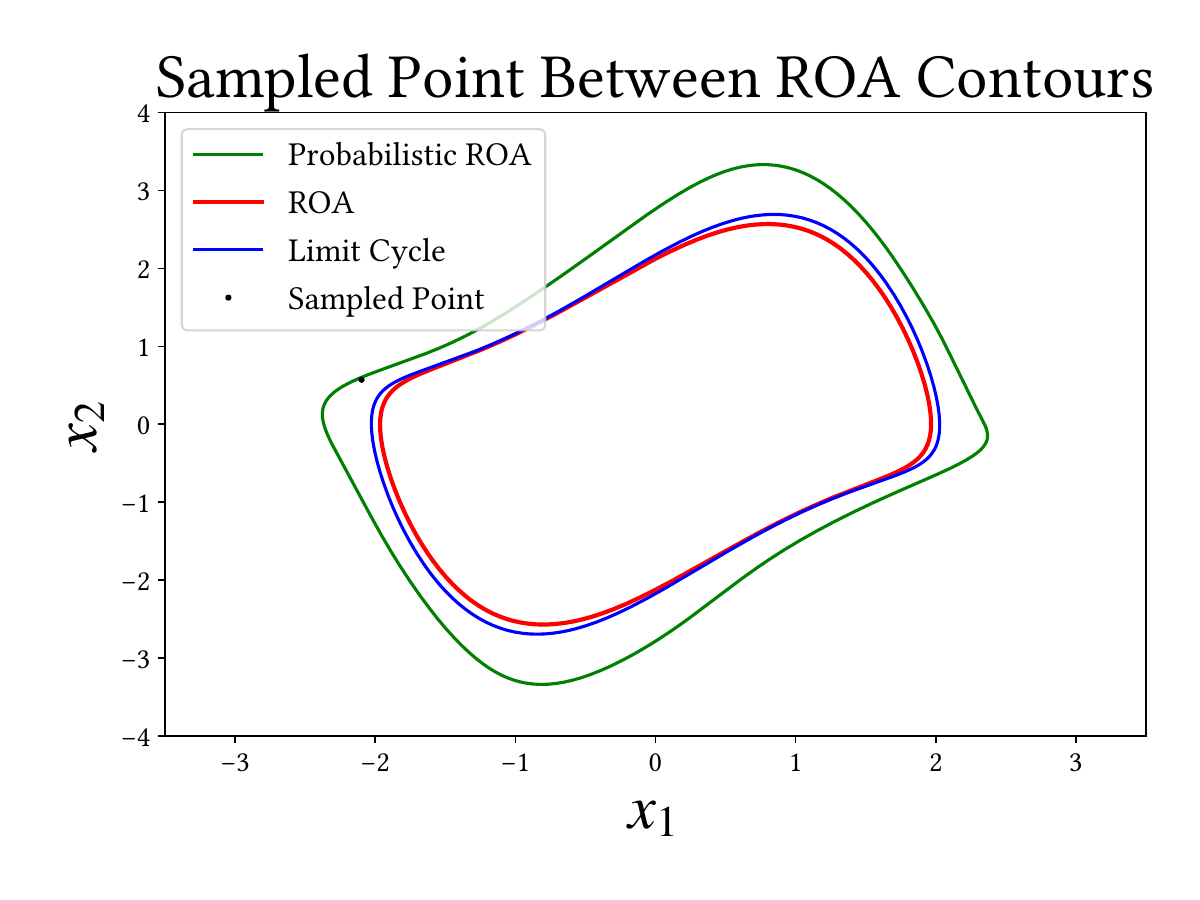}
    \caption{The probabilistic ROA Heatmap visualizes the attraction probability for the initial point in \(\mathcal{C}\). Comparison between ROA and \(\mathcal{C}\) indicates that, under certain conditions, the probabilistic ROA can extend beyond the limit cycle for deterministic system, suggesting that stochastic effects may stabilize the system. }
    \label{fig: p-roa heatmeap}
\end{figure}
\end{eg}

Based on the verified level sets shown in Figure \ref{fig: nn_LF and level sets} and by applying Proposition \ref{prop:p-roa}, we generate a heatmap in the left image of Figure \ref{fig: p-roa heatmeap} to visualize the attraction probability for the initial point in \(\mathcal{C}\).
As depicted in the right panel of Figure \ref{fig: p-roa heatmeap}, the noise term in the stochastic differential equation can sometimes enlarge the region of attraction.  In our experiments, we sampled points from within the probabilistic ROA while excluding those on the deterministic limit cycle (indicated by the blue curve, which marks the boundary of the deterministic ROA). We then used this sampling point as the initial value and simulate 10,000 trajectories from it to estimate the attraction probability from this point. The results show that over $10\%$ of the trajectories converged. This indicates that random disturbances can, under certain conditions, have a stabilizing effect on the stochastic system.



\section{Conclusions and Future Works}

We employ a data-driven PINN to approximate the solution of stochastic Zubov PDE for stochastic system and use dreal to formally verify the stochastic Lyapunov condition for the learned neural Lyapunov function both locally and over the largest probabilistic region of attraction $\mathcal{C}$. We also derive sufficient stochastic Lyapunov conditions for probabilistic attraction and propose an attraction probability function based on both the local and neural Lyapunov functions. In our experiments, we generate a heatmap to visualize the attraction probability. Our results demonstrate that the neural approach yields a larger verified probabilistic ROA than the Lyapunov function computed by the SOS method.

Future research will focus on improving scalability for high-dimensional systems. A significant challenge remains in addressing high-dimensional dynamic systems, which are inherently difficult to solve even with neural networks. This difficulty is particularly pronounced during training data generation and verification, as increased dimensionality demands substantially more computational resources and time. Another promising direction is to try other verification tools, e.g., \cite{zhou2024scalable}, which would help enlarge the probabilistic region of attraction \(\mathcal{C}\). Additionally, Proposition \ref{prop:p-roa} only provides a conservative estimate of the $p$-ROA. It would be of interest to investigate how to optimize the attraction probabilities. 

\section{ACKNOWLEDGMENTS}

The authors are grateful to Dr. Yiming Meng for helpful discussions.


\addtolength{\textheight}{-3cm}   


\bibliographystyle{unsrt}
\bibliography{references}

\begin{thebibliography}{10}

\bibitem{chang2019neural}
Ya-Chien Chang, Nima Roohi, and Sicun Gao.
\newblock Neural {L}yapunov control.
\newblock {\em Advances in NeurIPS}, 32, 2019.

\bibitem{zhou2022neural}
Ruikun Zhou, Thanin Quartz, Hans De~Sterck, and Jun Liu.
\newblock Neural {L}yapunov control of unknown nonlinear systems with stability guarantees.
\newblock {\em Advances in NeurIPS}, 35, 2022.

\bibitem{vournas2006region}
Costas~D Vournas and Nikos~G Sakellaridis.
\newblock Region of attraction in a power system with discrete {LTC}s.
\newblock {\em IEEE Transactions on Circuits and Systems I: Regular Papers}, 53(7):1610--1618, 2006.

\bibitem{manchester2011regions}
Ian~R Manchester, Mark~M Tobenkin, Michael Levashov, and Russ Tedrake.
\newblock Regions of attraction for hybrid limit cycles of walking robots.
\newblock {\em IFAC Proceedings Volumes}, 44(1):5801--5806, 2011.

\bibitem{matthews2012region}
Megan~Leigh Matthews and Cranos~M Williams.
\newblock Region of attraction estimation of biological continuous boolean models.
\newblock In {\em Proc. of IEEE SMC}, pages 1700--1705. IEEE, 2012.

\bibitem{vannelli1985maximal}
Anthony Vannelli and Mathukumalli Vidyasagar.
\newblock Maximal {L}yapunov functions and domains of attraction for autonomous nonlinear systems.
\newblock {\em Automatica}, 21(1):69--80, 1985.

\bibitem{liu2023towards}
Jun Liu, Yiming Meng, Maxwell Fitzsimmons, and Ruikun Zhou.
\newblock Towards learning and verifying maximal neural {L}yapunov functions.
\newblock In {\em Proc. of CDC}, pages 8012--8019. IEEE, 2023.

\bibitem{liu2025physics}
Jun Liu, Yiming Meng, Maxwell Fitzsimmons, and Ruikun Zhou.
\newblock Physics-informed neural network {L}yapunov functions: {PDE} characterization, learning, and verification.
\newblock {\em Automatica}, 175:112193, 2025.

\bibitem{kang2023data}
Wei Kang, Kai Sun, and Liang Xu.
\newblock Data-driven computational methods for the domain of attraction and {Z}ubov's equation.
\newblock {\em IEEE Transactions on Automatic Control}, 69(3):1600--1611, 2023.

\bibitem{abate2020formal}
Alessandro Abate, Daniele Ahmed, Mirco Giacobbe, and Andrea Peruffo.
\newblock Formal synthesis of {L}yapunov neural networks.
\newblock {\em IEEE Control Systems Letters}, 5(3):773--778, 2020.

\bibitem{gaby2022lyapunov}
Nathan Gaby, Fumin Zhang, and Xiaojing Ye.
\newblock Lyapunov-net: A deep neural network architecture for lyapunov function approximation.
\newblock In {\em Proc. of CDC}, pages 2091--2096. IEEE, 2022.

\bibitem{ravanbakhsh2019learning}
Hadi Ravanbakhsh and Sriram Sankaranarayanan.
\newblock Learning control {L}yapunov functions from counterexamples and demonstrations.
\newblock {\em Autonomous Robots}, 43:275--307, 2019.

\bibitem{liu2024tool}
Jun Liu, Yiming Meng, Maxwell Fitzsimmons, and Ruikun Zhou.
\newblock Tool {LyZNet}: A lightweight python tool for learning and verifying neural {L}yapunov functions and regions of attraction.
\newblock In {\em Proc. of HSCC}, pages 1--8, 2024.

\bibitem{dawson2023safe}
Charles Dawson, Sicun Gao, and Chuchu Fan.
\newblock Safe control with learned certificates: A survey of neural lyapunov, barrier, and contraction methods for robotics and control.
\newblock {\em IEEE Transactions on Robotics}, 39(3):1749--1767, 2023.

\bibitem{zubov1961methods}
Vladimir~Ivanovich Zubov.
\newblock {\em Methods of {AM} {L}yapunov and their application}, volume 4439.
\newblock US Atomic Energy Commission, 1961.

\bibitem{camilli2000zubov}
Fabio Camilli and Paola Loreti.
\newblock A characterization of the domain of attraction for a locally exponentially stable stochastic system.
\newblock {\em Quaderno IAC}, n:25, 2000.

\bibitem{camilli2002characterizing}
Fabio Camilli and Lars Gr{\"u}ne.
\newblock Characterizing attraction probabilities via the stochastic zubov equation.
\newblock {\em Discrete and Continuous Dynamical Systems-B}, 3(3):457--468, 2003.

\bibitem{camilli2005zubov}
Fabio Camilli, Lars Gr{\"u}ne, and Fabian Wirth.
\newblock Zubov's method for stochastic control systems.
\newblock {\em IFAC Proceedings Volumes}, 38(1):259--264, 2005.

\bibitem{hafstein2018lyapunov}
Sigurdur Hafstein, Skuli Gudmundsson, Peter Giesl, Enrico Scalas, et~al.
\newblock Lyapunov function computation for autonomous linear stochastic differential equations using sum-of-squares programming.
\newblock {\em Discrete and Continuous Dynamical Systems-B}, 23(2):939--956, 2018.

\bibitem{zhang2022neural}
Jingdong Zhang, Qunxi Zhu, and Wei Lin.
\newblock Neural stochastic control.
\newblock {\em Advances in NeurIPS}, 35, 2022.

\bibitem{lechner2022stability}
Mathias Lechner, {\DJ}or{\dj}e {\v{Z}}ikeli{\'c}, Krishnendu Chatterjee, and Thomas~A Henzinger.
\newblock Stability verification in stochastic control systems via neural network supermartingales.
\newblock In {\em Proc. of AAAI}, volume 36(7), pages 7326--7336, 2022.

\bibitem{ansaripour2022learning}
Matin Ansaripour, Mathias Lechner, {\DH}orde {\v{Z}}ikelic, Krishnendu Chatterjee, and Thomas~A Henzinger.
\newblock Learning control policies for region stabilization in stochastic systems.
\newblock {\em arXiv:2210.05304}, 2022.

\bibitem{gao2013dreal}
Sicun Gao, Soonho Kong, and Edmund~M Clarke.
\newblock d{R}eal: An {SMT} solver for nonlinear theories over the reals.
\newblock In {\em Proc. of CADE}, pages 208--214. Springer, 2013.

\bibitem{khasminskii2011stochastic}
Rafail Khasminskii.
\newblock {\em Stochastic stability of differential equations}, volume~66.
\newblock Springer Science \& Business Media, 2011.

\bibitem{mao2007stochastic}
Xuerong Mao.
\newblock {\em Stochastic Differential Equations and Applications}.
\newblock Elsevier, 2007.

\bibitem{camilli2001generalization}
Fabio Camilli, Lars Gr{\"u}ne, and Fabian Wirth.
\newblock A generalization of {Z}ubov's method to perturbed systems.
\newblock {\em SIAM Journal on Control and Optimization}, 40(2):496--515, 2001.

\bibitem{bjornsson2018local}
Hj{\"o}rtur Bj{\"o}rnsson, Peter Giesl, Skuli Gudmundsson, and Sigurdur~F Hafstein.
\newblock Local {L}yapunov functions for nonlinear stochastic differential equations by linearization.
\newblock In {\em Proc. of ICINCO}, pages 589--596, 2018.

\bibitem{meng2024stochastic}
Yiming Meng and Jun Liu.
\newblock Stochastic {L}yapunov-barrier functions for robust probabilistic reach-avoid-stay specifications.
\newblock {\em IEEE Transactions on Automatic Control}, 2024.

\bibitem{zhou2024scalable}
Duo Zhou, Christopher Brix, Grani~A Hanasusanto, and Huan Zhang.
\newblock Scalable neural network verification with branch-and-bound inferred cutting planes.
\newblock {\em arXiv:2501.00200}, 2024.

\end{thebibliography}
\end{document}